\documentclass[12pt,a4paper]{amsart}

\usepackage{mathrsfs}
\usepackage{amssymb}
\usepackage{enumerate}
\usepackage{mathptmx}

\usepackage{hyperref}
\hypersetup{colorlinks=true, citecolor=green, linkcolor=magenta, urlcolor=blue}
\usepackage{graphicx}
\usepackage{cite}
\usepackage{comment}

\newtheorem{theorem}{Theorem}[section]
\newtheorem{lemma}[theorem]{Lemma}
\newtheorem{corollary}[theorem]{Corollary}

\theoremstyle{definition}
\newtheorem{definition}[theorem]{Definition}

\theoremstyle{remark}
\newtheorem{remark}[theorem]{Remark}

\numberwithin{equation}{section}

\newcommand{\di}{n} 
\newcommand{\R}{\mathbb R} 
\newcommand{\N}{\mathbb N} 
\newcommand{\dom}{\Omega} 
\newcommand{\Dom}{D}
\newcommand{\bdy}{\partial^\infty\dom} 
\newcommand{\cl}{\overline{\dom}^\infty} 
\newcommand{\dist}{\delta_\dom} 
\newcommand{\diam}{\operatorname{diam}} 
\newcommand{\GF}{G_\dom} 
\newcommand{\LS}{L} 
\newcommand{\CF}{C} 
\newcommand{\Sob}{W} 
\newcommand{\DS}{\dot{W}} 
\newcommand{\cp}{\operatorname{cap}_2} 
\newcommand{\ms}{\mu} 
\newcommand{\pms}{\nu} 
\newcommand{\RM}{\mathcal{M}_+(\dom)} 
\newcommand{\bv}{f} 
\newcommand{\msol}{\underline{u}} 
\newcommand{\vmsol}{\underline{w}}
\newcommand{\BC}{\mathscr B}
\newcommand{\qtext}{\quad\text}
\newcommand{\con}{c} 

\makeatletter
\newcounter{con}
\def\cont{\refstepcounter{con}\con_{\thecon}}
\newcommand{\conlabel}[1]{\cont\label{#1}}
\newcommand{\conref}[1]{\con_{{\ref{#1}}}}
\makeatother

\title[Sublinear elliptic equations]{The Dirichlet problem
for sublinear elliptic equations with source}
\date{\today}

\author{Kentaro Hirata}
\address{Department of Mathematics, Graduate School of Advanced Science and Engineering,
Hiroshima University, Higashi-Hiroshima 739-8526, Japan}
\email{\href{mailto:hiratake@hiroshima-u.ac.jp}{hiratake@hiroshima-u.ac.jp}}
\thanks{K. H. is partially supported by JSPS KAKENHI Grant Number JP18K03333.}

\author{Adisak Seesanea}
\address{Nonlinear Analysis Unit, Okinawa Institute of Science and Technology Graduate University,
1919-1 Tancha, Onna-son, Kunigami-gun, Okinawa 904-0495, Japan}
\email{\href{mailto:adisak.seesanea@oist.jp}{adisak.seesanea@oist.jp}}
\thanks{A. S. is partially supported by JSPS KAKENHI Grant Number JP17H01092.}

\begin{document}
\maketitle

\begin{abstract}
We present a necessary and sufficient condition
on nonnegative Radon measures $\mu$ and $\nu$
for the existence of a positive continuous solution
of the Dirichlet problem for
the sublinear elliptic equation
$-\Delta u=\mu u^q+\nu$
with prescribed nonnegative continuous boundary data
in a general domain.
Moreover, two-sided pointwise estimates of Brezis--Kamin type
for positive bounded solutions
and the uniqueness of a positive continuous $L^q$-solution
are investigated.
\end{abstract}

\section{Introduction}

This paper deals with
the existence, uniqueness and two-sided pointwise estimates of
positive continuous solutions of the Dirichlet problems
for sublinear elliptic equations
involving nonnegative Radon measures.

Unless otherwise explicitly stated,
we always assume that
$\dom$ is a domain in $\R^\di$ $(\di\ge2)$
possessing the Green function $\GF$
for the Dirichlet Laplacian $-\Delta$ and that
it is regular for the Dirichlet problem for the Laplace equation.
When $\dom\neq\R^\di$,
we write $\dist(x)$
for the Euclidean distance from a point $x\in\dom$
to the Euclidean boundary $\partial\dom$ of $\dom$.
The boundary and closure of $\dom$
in the one point compactification $\R^\di\cup\{\infty\}$
are denoted by $\bdy$ and $\cl$,
respectively.
By $\RM$,
we denote the set of all nonnegative Radon measures on $\dom$.
Also, $\CF_+(\bdy)$ stands for
the set of all nonnegative continuous functions on $\bdy$.

Let $\ms,\pms\in\RM$ and $\bv\in\CF_+(\bdy)$.
We consider the Dirichlet problem
\begin{equation}\label{eq:SubLE}
  \begin{cases}
	-\Delta u=\ms u^q+\pms & \text{in } \dom,  \\
	u>0 & \text{in } \dom,  \\
	u=\bv & \text{on } \bdy,
  \end{cases}
\end{equation}
in the sublinear case $0<q<1$. It will be understood in the integral sense:
\begin{equation}\label{eq:int-eq}
	u(x) = \GF[u^q\,d\ms](x)+\GF[\pms](x)+H_{\bv}(x)
  \qtext{for all } x\in\dom,
\end{equation}
with $H_{\bv}$ being the Perron--Wiener--Brelot solution of
the Dirichlet problem
\begin{equation}\label{eq:PWB}
  \begin{cases}
	\Delta h=0 & \text{in } \dom,  \\
	h=\bv & \text{on } \bdy,
  \end{cases}
\end{equation}
and $\GF[\psi\,d\ms]$ standing for
the Green potential of $\psi\,d\ms$ defined by
\[
	\GF[\psi\,d\ms](x) := \int_\dom \GF(x,y) \psi(y) \, d\ms(y)
  \qtext{for } x\in\dom.
\]
When $\psi\equiv 1$,
we simply write $\GF[\ms]$.

A variety of classes of solutions of the Dirichlet problem
for sublinear elliptic equations
with measurable coefficient $a(x)$ and zero boundary values
such as
\begin{equation}\label{eq:inhomo-SubLE}
  \begin{cases}
	-\Delta u=a(x) u^q & \text{in } \dom,  \\
	u>0 & \text{in } \dom,  \\
	u=0 & \text{on } \bdy,
  \end{cases}
\end{equation}
have been investigated by many authors.
See 
\cite{MR3744665, MR1141779, MR820658, MR3311903, MR3556326, MR3567503, MR2752440, MR2041704, MR3881877, MR3985926, MR4048382, MR3792109}
for example.
Among these,
M\^{a}agli and Masmoudi \cite{MR2041704}
gave a sufficient condition for
the existence of a continuous solution of \eqref{eq:inhomo-SubLE}
in the case where
$\dom$ is an unbounded domain with
compact smooth boundary in $\R^\di$ $(\di\ge3)$.
They also gave two-sided pointwise estimates:
\[
	\con^{-1} \frac{\dist(x)}{\|x\|^{\di-1}}
  \le
	u(x)
  \le
	\con \GF[a](x)
  \qtext{for all } x\in\dom,
\]
where $\con>1$ is a constant.
In \cite{MR1141779},
Brezis and Kamin established a necessary and sufficient condition for
the existence of a bounded solution of \eqref{eq:inhomo-SubLE}
when $\dom=\R^\di$.
They also obtained the uniqueness result and 
the sharp pointwise estimates:
\[
	\con^{-1} G_{\R^\di}[a](x)^{\frac{1}{1-q}}
  \le
	u(x)
  \le
	\con G_{\R^\di}[a](x)
  \qtext{for all } x \in \R^\di.
\]
See also Cao and Verbitsky \cite{MR3556326,MR3567503},
who characterized the existence of a larger class of solutions,
satisfying Brezis--Kamin type estimates
in terms of Wolff potentials,
to similar homogeneous problems \eqref{eq:inhomo-SubLE}
with the $p$-Laplace and fractional Laplace operators in $\R^\di$.

Our first result on \eqref{eq:SubLE} is related to
\cite{MR2041704} and \cite{MR1141779},
but is new even when $\pms\equiv 0$,
or $\bv\equiv 0$,
or $\ms,\pms$ are measurable functions.
An important notion in our approach
is the following Kato type class of measures.
By $B(x,r)$, we denote the open ball of center $x$ and radius $r$.

\begin{definition}
Let $\ms\in\RM$.
We say that $\ms$ satisfies the {\em $\GF$-Kato condition} if
\begin{align}
	&\lim_{r\to0+}
	\biggl( \sup_{x\in\dom}
	\int_{\dom\cap B(x,r)} \GF(x,y)\, d\ms(y) \biggr)
  =
	0,  \label{eq:katocond-1}\\
	&\lim_{r\to0+}
	\biggl( \sup_{x\in\dom}
	\int_{\dom\setminus B(0,1/r)} \GF(x,y)\, d\ms(y) \biggr)
  =
	0.  \label{eq:katocond-2}
\end{align}
\end{definition}

Note that \eqref{eq:katocond-2} has no meaning
when $\dom$ is bounded.
Also, if $\ms$ satisfies the $\GF$-Kato condition,
then $\GF[\ms]\in\LS^\infty(\dom)$ (see Lemma \ref{lem:bdd-pot}).
The $\GF$-Kato condition was utilized earlier
by the first named author \cite{MR2386468}
in the study of the Dirichlet problem for
$\Delta u=a(x)u^{-p}$ with $p>0$.
In this work, we show that such $\GF$-Kato conditions characterize
continuity of the corresponding Green potentials (see Lemma \ref{lem:cont-Kato}), and consequently the existence
of continuous solutions of \eqref{eq:SubLE} as follows.

We say that $u$ is {\em minimal}
in a class $\mathcal{F}$ of positive functions in $\dom$
if $u\le v$ in $\dom$ for any $v\in\mathcal{F}$.
For a function $\psi$ defined on a set $E$,
we denote by $\|\psi\|_\infty$
the usual supremum norm of $\psi$ on $E$.

\begin{theorem}\label{thm:existence}
Let $0<q<1$.
Assume that $\bv\in\CF_+(\bdy)$ and
$\ms,\pms\in\RM$ satisfy
\begin{equation}\label{eq:cond-f-meas}
	\|\bv\|_\infty+\|\GF[\ms]\|_\infty+\|\GF[\pms]\|_\infty>0.
\end{equation}
Then \eqref{eq:SubLE} has a solution $u\in\CF(\cl)$
if and only if both of $\ms$ and $\pms$ satisfy the $\GF$-Kato condition.
Moreover, in this case,
\eqref{eq:SubLE} has a minimal solution in $\CF(\cl)$.
\end{theorem}

Condition \eqref{eq:cond-f-meas} cannot be removed
to ensure the existence of a positive solution.
For the sufficiency in Theorem \ref{thm:existence},
we present two different proofs,
based on the Schauder fixed point theorem and
the method of successive approximations,
with the help of equivalent conditions to the $\GF$-Kato condition,
a lower bound estimate for solutions and
an iterated inequality for Green potentials
established recently by Grigor'yan and Verbitsky \cite{MR4105916},
which are also useful to show the necessity part of Theorem \ref{thm:existence}.

Needless to say, the continuous solution of \eqref{eq:SubLE} obtained in Theorem \ref{thm:existence} is bounded. The next theorem yields in particular
the Brezis--Kamin type estimates of such a solution.

\begin{theorem}\label{thm:upper-lower-est}
Let $0<q<1$, let $\bv\in\CF_+(\bdy)$ and let
$\ms,\pms\in\RM$.
If $u$ is a bounded solution of \eqref{eq:SubLE}, it enjoys the following estimates.
\begin{enumerate}
\item Lower estimate:
\[
	u(x)
  \ge
	\conlabel{c:lb}\GF[\ms](x)^{\frac{1}{1-q}}+\GF[\nu](x)+H_\bv(x)
  \qtext{for all } x\in\dom.
\]
\item Upper estimate:
\[
	u(x)
  \le
	\conlabel{c:ub}^q\GF[\ms](x)+\GF[\nu](x)+H_\bv(x)
  \qtext{for all } x\in\dom.
\]
\item Uniform estimate:\:
$\|u\|_\infty \le\conref{c:ub}$.
\end{enumerate}
Here $\conref{c:lb}:=(1-q)^{\frac{1}{1-q}}$ and
\begin{equation}\label{eq:ub-def}
	\conref{c:ub}
  :=
	\max\left\{ 1,(\|\GF[\ms]\|_\infty+\|\GF[\pms]\|_\infty+\|\bv\|_\infty)^{\frac{1}{1-q}} \right\}<+\infty.
\end{equation}
\end{theorem}

\begin{remark}
Note that if \eqref{eq:SubLE} has a bounded solution,
then $\GF[\ms]$ and $\GF[\pms]$ must be in $\LS^\infty(\dom)$
in light of \eqref{eq:int-eq} and Lemma \ref{lem:LB-1}.
\end{remark}

In our present study,
we are also interested in integrability properties of
continuous solutions of \eqref{eq:SubLE}.
The lower estimate in Theorem \ref{thm:upper-lower-est}
gives a hint of conditions for the existence of a solution in
$\CF(\cl)\cap\LS^\gamma(\dom,d\ms)$.
This is related to Verbitsky \cite{MR3792109},
where the existence of a solution in
$\LS^\gamma(\dom,d\ms)$ of \eqref{eq:SubLE}
was characterized
in terms of $\GF[\ms]\in\LS^{\frac{\gamma}{1-q}}(\dom,d\ms)$
in the case where $\bv\equiv 0$ and $\pms\equiv 0$,
and to \cite{MR3881877,MR4048382},
where the second named author and Verbitsky
characterized the existence and uniqueness of a solution 
in $\dot{\Sob}_0^{1,2}(\dom)\cap\LS^q_{\rm loc}(\dom,d\ms)$
of \eqref{eq:SubLE} with $\bv\equiv0$
in terms of
\begin{equation*}\label{eq:finite-energy}
	\GF[\ms]\in\LS^{\frac{1+q}{1-q}}(\dom,d\ms)
  \quad\text{and}\quad
	\GF[\pms]\in\LS^{1+q}(\dom,d\ms).
\end{equation*}
Here $\DS^{1,2}_0(\dom)$
is the homogeneous Sobolev (or Dirichlet) space defined as
the closure of $\CF_0^{\infty}(\dom)$ with respect to the seminorm 
\begin{equation}\label{eq:DS-norm}
	\|u\|_{\DS^{1,2}_0(\dom)}
  :=
	\| \nabla u \|_{\LS^2(\dom)}.
\end{equation}
For \eqref{eq:SubLE} with $\bv\equiv0$,
it was further shown that
the existence of a solution in
$\DS^{1,2}_{0}(\dom)\cap\LS^q_{\rm loc}(\dom,d\ms)$
is equivalent to the existence of a solution in
$\LS^{1+q}(\dom, d\ms)$.
Note that solutions $u$ of \eqref{eq:SubLE}
with $\ms,\pms$ satisfying the $\GF$-Kato condition
do not always satisfy $\| \nabla u \|_{\LS^2(\dom)}<+\infty$
(see Corollary \ref{exm:kato-2}).

\begin{theorem}\label{thm:L^q-sol}
Under the same assumptions as in Theorem \ref{thm:existence}.
Let $\gamma>q$.
Then \eqref{eq:SubLE} has a solution
$u\in\CF(\cl)\cap\LS^\gamma(\dom,d\ms)$
if and only if both of $\ms$ and $\pms$ satisfy the $\GF$-Kato condition
and
\begin{equation}\label{eq:cond-L^q-sol}
	\GF[\ms]\in\LS^{\frac{\gamma}{1-q}}(\dom,d\ms),
  \quad
	\GF[\nu]+H_{\bv}\in\LS^\gamma(\dom,d\ms).
\end{equation}
Moreover, in this case, we have the following.
\begin{enumerate}
\item
Problem \eqref{eq:SubLE} has a minimal solution in
$\CF(\cl)\cap\LS^\gamma(\dom,d\ms)$.

\item
Assume that either $\gamma=q$ or $\gamma=q+1$ and $\bv\equiv0$
holds.
Then \eqref{eq:SubLE} has exactly one solution in
$\CF(\cl)\cap\LS^\gamma(\dom,d\ms)$.
\end{enumerate}
\end{theorem}

This paper is organized as follows.
In Section \ref{sec:pre},
we provide equivalent conditions to the $\GF$-Kato condition
and prove the equicontinuity of a family of certain Green potentials,
which is one of key tools in our approaches.
In Section \ref{sec:ULest},
we prove Theorem \ref{thm:upper-lower-est}.
Two different proofs of Theorem \ref{thm:existence}
are given in Section \ref{sec:existence}.
In Section \ref{sec:L^q-sol},
we address the uniqueness of a positive continuous $\LS^q$-solution and
prove Theorem \ref{thm:L^q-sol}.
In the final section \ref{sec:remark},
we give supplementary remarks on
H\"older continuous solutions and
the existence of a positive $\CF^1$-solution
whose Dirichlet integral diverges.

\section{Preliminaries}\label{sec:pre}

We use the symbol $\con$
to denote an absolute positive constant
whose value may vary at each occurrence.
For two positive functions $f$ and $g$,
we write $f\lesssim g$
if $f(x)\le \con g(x)$ for some positive constant $\con$
independent of variable $x$.

In the study of the stationary Schr\"odinger equation
$\Delta u=\ms u$ in $\R^\di$,
Boukricha--Hansen--Hueber \cite{MR887788} gave
necessary and sufficient conditions
on $\ms\in\mathcal{M}_+(\R^\di)$
for $G_{\R^\di}[\mu]\in\CF(\R^\di)$.
We modify their arguments to
provide equivalent conditions to the $\GF$-Kato condition.
For the sake of convenience,
we make the notational convention that
$B(z,r)=(\R^\di\cup\{\infty\})\setminus\overline{B(0,1/r)}$ and
\begin{equation}\label{eq:conven}
	\int_{\dom\cap B(z,r)} \GF(x,y) \, d\ms(y)
  =
	\int_{\dom\setminus\overline{B(0,1/r)}} \GF(x,y) \, d\ms(y),
\end{equation}
when $z=\infty$.
Since $\GF(\cdot,y)$ vanishes continuously on $\bdy$ for each $y\in\dom$
by the assumption that
$\dom$ is regular for the Dirichlet problem,
we note that
$\GF[\ms]$ is extended to $\cl$
as a lower semicontinuous function on $\cl$,
by assigning its value $0$ on $\bdy$.

\begin{lemma}\label{lem:equivKato}
Let $\ms\in\RM$. Then the following are equivalent:
\begin{enumerate}[\quad\rm (a)]\itemsep=2pt
\item
$\ms$ satisfies the $\GF$-Kato condition,
\item
$\displaystyle
\lim_{r\to0+} \biggl( \sup_{x\in\dom}
\int_{\dom\cap B(z,r)} \GF(x,y) \, d\ms(y) \biggr)
=0$\,
for any $z\in\cl$,
\item
$\displaystyle
\lim_{r\to0+} \biggl( \sup_{z\in\cl}
\biggl( \sup_{x\in\dom}
\int_{\dom\cap B(z,r)} \GF(x,y) \, d\ms(y) \biggr) \biggr)
=0$.
\end{enumerate}
\end{lemma}

\begin{proof}
(a) $\Rightarrow$ (b).\:
Let $z\in\overline{\dom}$.
We see from the domination principle
(\cite[Theorem 5.1.11]{MR1801253}) that
\begin{align*}
	\sup_{x\in\dom} \int_{\dom\cap B(z,r)} \GF(x,y) \, d\ms(y)
  &=
	\sup_{x\in\dom\cap B(z,r)} \int_{\dom\cap B(z,r)} \GF(x,y) \, d\ms(y)  \\
  &\le
	\sup_{x\in\dom\cap B(z,r)} \int_{\dom\cap B(x,2r)} \GF(x,y) \, d\ms(y).
\end{align*}
By \eqref{eq:katocond-1}, the right hand side goes to zero as $r\to0+$.
The case $z=\infty$ of (b) is equivalent to \eqref{eq:katocond-2}
by our convention \eqref{eq:conven}.
Hence (b) follows.

(b) $\Rightarrow$ (c).\:
We only consider the case where $\dom$ is unbounded.
An argument below can be modified easily to the case of bounded domains.
Let $\varepsilon>0$.
For each $z\in\cl$, we find $r_z>0$ such that
\[
	\sup_{x\in\dom} \int_{\dom\cap B(z,2r_z)} \GF(x,y) \, d\ms(y) < \varepsilon.
\]
Since $\cl$ is compact,
there exist  $z_1,\dots,z_{\ell-1}\in\dom$ and $z_\ell=\infty$ such that
$\cl\subset \bigcup_{i=1}^\ell B(z_i,r_i)$,
where $r_i:=r_{z_i}$ and
$B(z_\ell,r_\ell):=(\R^\di\cup\{\infty\})\setminus\overline{B(0,1/r_\ell)}$.
Let $\delta:=\min\{r_1,\dots,r_\ell,1/(2r_\ell)\}$
and let $w\in\cl$.
Then there is $j\in\{1,\dots,\ell\}$ such that
$w\in B(z_j,r_j)$.
If $j\in\{1,\dots,\ell-1\}$, then
$B(w,\delta)\subset B(w,r_j)\subset B(z_j,2r_j)$,
and so
\[
	\sup_{x\in\dom} \int_{\dom\cap B(w,\delta)} \GF(x,y) \, d\ms(y)
  \le
	\sup_{x\in\dom} \int_{\dom\cap B(z_j,2r_j)} \GF(x,y) \, d\ms(y)
  < \varepsilon.
\]
If $j=\ell$, then
$B(w,\delta)\subset (\R^\di\cup\{\infty\})\setminus\overline{B(0,1/(2r_\ell))}=B(z_\ell,2r_\ell)$, and so
\[
	\sup_{x\in\dom} \int_{\dom\cap B(w,\delta)} \GF(x,y) \, d\mu(y)
  \le
	\sup_{x\in\dom} \int_{\dom\cap B(z_\ell,2r_\ell)} \GF(x,y) \, d\mu(y)
  < \varepsilon.
\]
These yield that
\[
	\sup_{w\in\cl}\biggl( \sup_{x\in\dom}
	\int_{\dom\cap B(w,\delta)} \GF(x,y) \, d\ms(y) \biggr)
  < \varepsilon,
\]
and hence (c) follows.

(c) $\Rightarrow$ (a).\:
This follows by taking $x=z\in\dom$ and $z=\infty$.
\end{proof}

As above, the finite covering argument yields the following.

\begin{lemma}\label{lem:bdd-pot}
Let $\ms\in\RM$ satisfy the $\GF$-Kato condition.
Then $\GF[\ms]\in\LS^\infty(\dom)$.
\end{lemma}

The restriction of $\ms$ to a Borel set $E$ is denoted by $\ms|_E$.

\begin{lemma}\label{lem:cont-Kato}
Let $\ms\in\RM$ and $\conlabel{c:equicont}>0$.
Then the following are equivalent:
\begin{enumerate}[\quad\rm (a)]
\item
$\mu$ satisfies the $\GF$-Kato condition,
\item
$\GF[\ms|_E]\in\CF(\cl)$ for any Borel subset $E$ of $\dom$,
\item
$\GF[\ms]\in\CF(\cl)$,
\item
the family $\{\GF[\psi\,d\ms]:\psi\in\Psi\}$ is equicontinuous on $\cl$, where
\[
	\Psi
  :=
	\{\psi\in\LS^\infty(\dom):0\le \psi\le \conref{c:equicont} \text{ a.e. on } \dom \}.
\]
\end{enumerate}
Moreover, in this case, $\GF[\psi\, d\ms]$ vanishes continuously on $\bdy$
for each $\psi\in\Psi$.
\end{lemma}

\begin{proof}
(d) $\Rightarrow$ (c).\:
Trivial.

(c) $\Rightarrow$ (b).\:
As mentioned in front of Lemma \ref{lem:equivKato},
$\GF[\ms|_E]$ has a lower semicontinuous extension to $\cl$.
Also,
$\GF[\ms|_E]=\GF[\ms]-\GF[\ms|_{\dom\setminus E}]$
and the right hand side is upper semicontinuous on $\cl$.
Therefore $\GF[\ms|_E]\in\CF(\cl)$.

(b) $\Rightarrow$ (a).\:
Let $z\in\dom$.
Observe that $\ms(\{z\})=0$ since $\ms(\{z\})\GF(z,z)\le\GF[\ms](z)<+\infty$.
By the absolutely continuity of integrals, we have
\begin{equation}\label{eq:pr-cont-kato-1}
	\lim_{r\to0+} \int_{\dom\cap B(z,r)} \GF(x,y) \, d\mu(y) =0
  \qtext{for each } x\in\cl.
\end{equation}
This is true for $z\in\bdy$ by the same reasoning.
Since $\GF[\ms|_{\dom\cap B(z,r)}]\in\CF(\cl)$ for each $r>0$,
the Dini theorem ensures that
the convergence in \eqref{eq:pr-cont-kato-1} is uniform for $x\in\cl$.
Thus (a) follows from Lemma \ref{lem:equivKato}.

(a) $\Rightarrow$ (d).\:
We only consider the case $\dom\neq\R^\di$.
The case $\dom=\R^\di$ can be handled similarly
if replacing $\dist(z)$ to $1$ below.
Let $\psi\in\Psi$, let $z\in\dom$,
let $0<\kappa<1$ and let $\varepsilon>0$.
Applying the Harnack inequality to the positive harmonic function
$\GF(\cdot,y)$ on $\dom\setminus\{y\}$,
we see that there exists a function
$c:(0,1)\to(0,+\infty)$ satisfying
$c(\kappa)\to0$ as $\kappa\to0$, and
\[
	\frac{1}{1+c(\kappa)} \GF(z,y)
  \le
	\GF(x,y)
  \le
	\{1+c(\kappa)\} \GF(z,y)
\]
for all $x\in B(z,\kappa^2\dist(z))$ and
$y\in\dom\setminus B(z,\kappa\dist(z))$.
Therefore, for all $x\in B(z,\kappa^2\dist(z))$,
\[
	\int_{\dom\setminus B(z,\kappa\dist(z))}
	|\GF(x,y)-\GF(z,y)| \, d\ms(y)
  \le
	c(\kappa) \|\GF[\ms]\|_\infty,
\]
and so by Lemma \ref{lem:bdd-pot},
\[
	\int_{\dom\setminus B(z,\kappa\dist(z))}
	|\GF(x,y)-\GF(z,y)| \, d\ms(y)
  \le
	\frac{\varepsilon}{2},
\]
whenever $\kappa$ is small enough.
Also, it follows from Lemma \ref{lem:equivKato} that
\[
	\int_{B(z,\kappa\dist(z))} \GF(x,y) \, d\ms(y)
  \le
	\frac{\varepsilon}{4}
  \qtext{for all } x\in \dom,
\]
if $\kappa$ is sufficiently small.
These estimates imply that
\[
	|\GF[\psi\,d\ms](x)-\GF[\psi\,d\ms](z)|
  \le
	\conref{c:equicont} \int_\dom |\GF(x,y)-\GF(z,y)| \, d\ms(y)
  \le
	\conref{c:equicont} \varepsilon,
\]
whenever $x\in B(z,\kappa^2\dist(z))$ and
$\kappa$ is sufficiently small.
Thus $\GF[\psi\, d\ms]$ is continuous at $z$
uniformly for $\psi\in\Psi$.
If $z\in\partial\dom$,
then by Lemma \ref{lem:equivKato} there exists $r>0$ such that
\[
	\int_{\dom\cap B(z,r)} \GF(x,y) \, d\ms(y)
  <
	\varepsilon
  \qtext{for all } x\in\dom.
\]
Since $\GF[\mu|_{\dom\setminus B(z,r)}]\in\CF(\overline{\dom}\cap B(z,r))$
and vanishes at $z$,
we see that $\GF[\ms]$ vanishes continuously at $z$.
Hence $\GF[\psi\, d\ms](x)\to0$ uniformly for $\psi\in\Psi$
as $x\to z$.
In the same way, we can show that
$\GF[\psi\, d\ms](x)\to0$  uniformly for $\psi\in\Psi$
as $x\to \infty$.
Hence (d) is proved.
\end{proof}

\section{Proof of Theorem \ref{thm:upper-lower-est}}\label{sec:ULest}

The following lower estimate for supersolutions of
sublinear elliptic equations
and iterated inequality for Green potentials
established recently by Grigor'yan and Verbitsky \cite{MR4105916}
are helpful tools in our approaches.

\begin{lemma}[{\cite[Theorem 1.3]{MR4105916}}]\label{lem:LB-1}
Let $0<q<1$ and $\ms\in\RM$.
If $u\in\LS^q_{\rm loc}(\dom,d\ms)$
is a positive solution of the integral inequality
$u\ge\GF[u^q\,d\ms]$ in $\dom$,
then
\[
	u(x) \ge (1-q)^{\frac{1}{1-q}} \GF[\ms](x)^{\frac{1}{1-q}}
  \qtext{for all } x\in\dom.
\]
\end{lemma}

\begin{lemma}[{\cite[Lemma 2.5]{MR4105916}}]\label{lem:iterateGF}
Let $\ms\in\RM$ and $s\ge1$. Then 
\[
	\GF[\ms](x)^s \le s \GF[\GF[\ms]^{s-1}\,d\ms](x)
  \qtext{for all } x\in\dom.
\]
\end{lemma}

\begin{proof}[Proof of Theorem \ref{thm:upper-lower-est}]
Let $u$ be any bounded solution of \eqref{eq:SubLE}.
Then
\[
	u(x)=\GF[u^q\,d\ms](x)+\GF[\pms](x)+H_\bv(x)
  \qtext{for } x\in\dom.
\]
Note from the maximum principle that
$H_\bv(x)\le \|\bv\|_\infty$ for $x\in\dom$.
If $\|u\|_\infty\ge1$, then
\begin{equation}\label{eq:upper-est-1}
\begin{split}
	\|u\|_\infty
  &\le
	\|u\|_\infty^q\|\GF[\ms]\|_\infty+\|\GF[\pms]\|_\infty+\|\bv\|_\infty  \\
  &\le
	\|u\|_\infty^q \left( \|\GF[\ms]\|_\infty+\frac{\|\GF[\pms]\|_\infty+\|\bv\|_\infty}{\|u\|_\infty^q} \right)  \\
  &\le
	\|u\|_\infty^q
	( \|\GF[\ms]\|_\infty+\|\GF[\pms]\|_\infty+\|\bv\|_\infty).
\end{split}
\end{equation}
Therefore we obtain the uniform estimate
$\|u\|_\infty\le\conref{c:ub}$.
Substituting this into
\[
	u(x) \le \|u\|_\infty^q \GF[\ms](x)+\GF[\pms](x)+H_\bv(x)
  \qtext{for } x\in\dom,
\]
we can get the upper estimate.
Thanks to Lemmas \ref{lem:LB-1} and \ref{lem:iterateGF},
we get the lower estimate
\begin{equation}\label{eq:lower-est-1}
\begin{split}
	u(x)
  &\ge
	(1-q)^{\frac{q}{1-q}} \GF[\GF[\ms]^{\frac{q}{1-q}}\,d\ms](x)
  +
	\GF[\pms](x) + H_\bv(x)  \\
  &\ge
	(1-q)^{\frac{q}{1-q}} (1-q) \GF[\ms](x)^{\frac{1}{1-q}}
  +
	\GF[\pms](x) + H_\bv(x),
\end{split}
\end{equation}
as desired.
\end{proof}

\section{Proof of Theorem \ref{thm:existence}}\label{sec:existence}

First, we give a proof of the necessity in Theorem \ref{thm:existence}.
For the sufficiency, we present two different proofs based on
the Schauder fixed point theorem and the method of successive approximations.

\subsection{A proof of the necessity}

\begin{proof}[Proof of Theorem \ref{thm:existence} (necessity)]
Suppose that \eqref{eq:SubLE} has a solution $u\in\CF(\cl)$.
Then $\GF[u^q\,d\ms]=u-\GF[\pms]-H_{\bv}$ on $\cl$.
Since the right hand side is upper semicontinuous on $\cl$,
we have $\GF[u^q\,d\ms],\GF[\pms]\in\CF(\cl)$.
Appealing to Lemma \ref{lem:cont-Kato}, the measures
$u^q\,d\ms$ and $\pms$ satisfy the $\GF$-Kato condition.
In order to show that $\ms$ satisfies the $\GF$-Kato condition,
we need an additional argument.
Let $z\in\cl$ and $r>0$.
For the sake of simplicity, we write
$\omega_r:=\ms|_{\dom\cap B(z,r)}$.
Using Lemmas \ref{lem:LB-1} and \ref{lem:iterateGF},
we get for all $x\in\dom$,
\begin{align*}
	\GF[u^q\,d\omega_r](x)
  &\ge
	(1-q)^{\frac{q}{1-q}}
	\GF[\GF[\ms]^{\frac{q}{1-q}}\,d\omega_r](x)  \\
  &\ge
	(1-q)^{\frac{q}{1-q}}
	\GF[\GF[\omega_r]^{\frac{q}{1-q}}\,d\omega_r](x)  \\
  &\ge
	(1-q)^{\frac{1}{1-q}}
	\GF[\omega_r](x)^{\frac{1}{1-q}}.
\end{align*}
Since $u^q\,d\ms$ satisfies the $\GF$-Kato condition as mentioned above,
we have
\[
	\lim_{r\to0+} \biggl( \sup_{x\in\dom} \GF[\omega_r](x) \biggr)=0.
\]
Therefore, by Lemma \ref{lem:equivKato},
$\ms$ satisfies the $\GF$-Kato condition as well.
\end{proof}

\subsection{A proof based on the Schauder fixed point theorem}\label{subsec:ex-1}

In the rest of this section,
we assume that
$\ms,\pms\in\RM$ satisfy the $\GF$-Kato condition.
We note from Lemma \ref{lem:bdd-pot} that
$\GF[\ms],\GF[\pms]\in\LS^\infty(\dom)$.

Taking Theorem \ref{thm:upper-lower-est} into account,
we consider the following function class.
Let 
\[
	w_0(x)
  :=
	(1-q)^{\frac{1}{1-q}}\GF[\ms](x)^{\frac{1}{1-q}}
  + \GF[\nu](x) + H_\bv(x)
  \qtext{for } x\in\dom.
\]
We see from 
\eqref{eq:cond-f-meas} and the minimum principle
that $w_0$ is positive on $\dom$.
With the constant $\conref{c:ub}$ defined by \eqref{eq:ub-def},
we let
\[
	\BC
  :=
	\left\{ w\in\CF(\cl) :
	w_0(x) \le w(x) \le \conref{c:ub}
	\text{ for } x\in\dom \right\},
\]
and consider the operator on $\BC$ defined by
\[
	T_\bv[w](x)
  :=
  \begin{cases}
	\GF[w^q\,d\ms](x) + \GF[\pms](x) + H_\bv(x) & \text{for } x\in\dom, \\
	\bv(x) & \text{for } x\in\bdy.
  \end{cases}
\]
We equip $\CF(\cl)$ with the uniform norm $\|\cdot\|_\infty$,
so that it is complete.
Also, we see that $\BC$ is closed and convex in $\CF(\cl)$.

\begin{lemma}\label{lem:FintoV}
$T_\bv[\BC]\subset \BC$.
\end{lemma}

\begin{proof}
Let $w\in \BC$.
Since $\dom$ is regular for the Dirichlet problem,
we have $H_\bv\in\CF(\cl)$ and
$H_\bv(x)\to\bv(\xi)$ as $x\to\xi\in\bdy$.
This and Lemma \ref{lem:cont-Kato} derive $T_\bv[w]\in\CF(\cl)$.
Also, since
$w(x)\ge w_0(x)\ge (1-q)^{\frac{1}{1-q}}\GF[\ms](x)^{\frac{1}{1-q}}$,
we can obtain
$w_0(x)\le T_\bv[w](x)\le\conref{c:ub}$
for all $x\in\dom$
by repeating similar arguments to \eqref{eq:upper-est-1}
and \eqref{eq:lower-est-1}.
Thus the lemma follows.
\end{proof}

\begin{lemma}\label{lem:relcompact}
$T_\bv[\BC]$ is relatively compact in $\CF(\cl)$.
\end{lemma}

\begin{proof}
Observe from $H_\bv\in\CF(\cl)$ and
Lemma \ref{lem:cont-Kato} that
$T_\bv[\BC]$ is equicontinuous on $\cl$.
Since $T_\bv[\BC]$ is uniformly bounded
by Lemma \ref{lem:FintoV},
this lemma follows from
the Ascoli--Arzel\'a theorem.
\end{proof}

\begin{lemma}\label{lem:cont-F}
$T_\bv$ is continuous on $\BC$.
\end{lemma}

\begin{proof}
Let $w_1,w_2\in \BC$.
Using the inequality
\[
	|a^q-b^q| \le |a-b|^q
  \qtext{for all $a,b\ge0$,} 
\]
we have for all $x\in\dom$,
\begin{align*}
	|T_\bv[w_1](x)-T_\bv[w_2](x)|
  &\le
	\int_\dom \GF(x,y)|w_1(y)-w_2(y)|^q \, d\ms(y)  \\
  &\le
	\|\GF[\ms]\|_\infty \|w_1-w_2\|_\infty^q.
\end{align*}
Taking the supremum for $x\in\dom$ on the left hand side
yields the continuity of $T_\bv$ on $\BC$.
\end{proof}

\begin{proof}[Proof of Theorem \ref{thm:existence} (sufficiency)]
By virtue of Lemmas \ref{lem:FintoV}, \ref{lem:relcompact} and \ref{lem:cont-F},
we can apply the Schauder fixed point theorem
to get $u\in \BC$ satisfying
$T_\bv[u]=u$ on $\cl$.
Since $u(x)\ge w_0(x)>0$ for all $x\in\dom$,
this $u$ is the desired one.
\end{proof}

\subsection{A proof based on the method of successive approximations}\label{subsec:ex-2}

An idea of the following proof is based on
Verbitsky \cite{MR3792109} and
the second named author and Verbitsky \cite{MR3881877}.

\begin{proof}[Proof of Theorem \ref{thm:existence} (sufficiency)]
We define a sequence $\{u_j\}$ inductively by
\begin{align*}
	u_0(x)
  &:=
	(1-q)^{\frac{1}{1-q}} \GF[\ms](x)^{\frac{1}{1-q}},  \\
	u_{j}(x)
  &:=
	\GF[u_{j-1}^q\,d\ms](x)+\GF[\pms](x)+H_\bv(x)
  \qtext{for } j\in\N.
\end{align*}
Then $u_1>0$ in $\dom$ by \eqref{eq:cond-f-meas}.
Also, by Lemma \ref{lem:iterateGF}, we have for all $x\in\dom$,
\begin{align*}
	u_1(x)
  &\ge
	\GF[u_0^q\,d\ms](x)
  =
	(1-q)^{\frac{q}{1-q}} \GF[\GF[\ms]^{\frac{q}{1-q}}\,d\ms](x)  \\
  &\ge
	(1-q)^{\frac{q}{1-q}} (1-q) \GF[\ms](x)^{\frac{1}{1-q}}
=
	u_0(x).
\end{align*}
By induction, we see that 
$\{u_j\}$ is nondecreasing.
Also,
\[
	\|u_j\|_\infty
  \le
	\|u_{j-1}\|_\infty^q \|\GF[\ms]\|_\infty
  + \|\GF[\pms]\|_\infty + \|\bv\|_\infty
  \qtext{for } j\in\N.
\]
Since $\GF[\ms],\GF[\pms]\in\LS^\infty(\dom)$
by Lemma \ref{lem:bdd-pot},
this implies that $\{u_j\}\subset\LS^\infty(\dom)$ and
\[
	\|u_j\|_\infty
  \le
	\|u_j\|_\infty^q\|\GF[\ms]\|_\infty
  +
	\|\GF[\pms]\|_\infty + \|\bv\|_\infty,
\]
which derives $\|u_j\|_\infty\le\conref{c:ub}$
by the same way as in \eqref{eq:upper-est-1}.
Therefore $\{u_j\}$ converges pointwisely to a positive $u\in\LS^\infty(\dom)$.
The monotone convergence theorem shows that
$u(x)=\GF[u^q\,d\ms](x)+\GF[\pms](x)+H_\bv(x)$
for all $x\in\dom$,
which also yields $u\in\CF(\cl)$ and
$u=f$ on $\bdy$
by Lemma \ref{lem:cont-Kato}.
This completes the proof.
\end{proof}

\begin{remark}\label{rem:minimalsol}
The solution $u$ obtained in the above proof is minimal
in the sense that $u\le v$ in $\dom$
for any solution $v\in\CF(\cl)$ of \eqref{eq:SubLE}.
In fact, $u_0\le v$ in $\dom$ by Lemma \ref{lem:LB-1},
and, in view of the definition of $u_j$,
we can get inductively
that $u_j\le v$ in $\dom$.
\end{remark}

\section{Proof of Theorem \ref{thm:L^q-sol}}\label{sec:L^q-sol}

We first recall some basic notion of Sobolev spaces
(see \cite{MR2305115, MZ} for example). 
The Sobolev space $\Sob^{1,2}(\dom)$ consists of all functions
$u \in\LS^2(\dom)$ such that $\|\nabla u \|\in\LS^2(\dom)$,
where $\nabla u$ is the vector of weak partial derivatives
of $u$ of order 1.
The corresponding local space
$\Sob_{\rm loc}^{1,2}(\dom)$ is the set of all functions
$u$ in $\dom$ such that the restriction
$u|_{\Dom}\in\Sob^{1, 2}(\Dom)$
for every relatively compact open subset $\Dom$ of $\dom$. 
By $\Sob^{1,2}_0(\dom)$ and $\DS^{1,2}_0(\dom)$,
we denote the closures of
$\CF^{\infty}_0(\dom)$ with respect to the usual Sobolev norm
and the seminorm \eqref{eq:DS-norm},
respectively.
Then
$\Sob^{1,2}_0(\dom)\subset\DS^{1,2}_{0}(\dom)\subset \Sob_{\rm loc}^{1,2}(\dom)$.
It is known that
$u\in\Sob_{\rm loc}^{1,2}(\dom)$
has a quasicontinuous representative $\widetilde{u}$,
namely, $\widetilde{u}=u$ a.e. in $\dom$ and,
for each $\varepsilon>0$,
there is an open set $\mathcal{O}\subset\dom$
such that $\cp(\mathcal{O})<\varepsilon$ and the restriction 
$\widetilde{u}|_{\dom \setminus \mathcal{O}}$
is finite and continuous on 
$\dom\setminus \mathcal{O}$. 

Let $\DS^{-1,2}(\dom)$
denote the dual space to $\DS_{0}^{1,2}(\dom)$.
The following result is due to Brezis and Browder \cite{BrB} (cf. \cite[Theorem 2.39]{MZ}).

\begin{lemma}[\cite{BrB}]\label{lem:BreB} 
Let $\mu\in\DS^{-1,2}(\dom) \cap \RM$ and 
$u\in\DS_{0}^{1,2}(\dom)$.
Then $\widetilde{u}\in\LS^{1}(\dom,d\mu)$ and 
\[
	\langle \mu, u\rangle = \int_{\dom} \widetilde{u} \, d\mu.
\]
\end{lemma}

\begin{lemma}[{\cite[Lemma 5.4]{MR4048382}}]\label{lem:dual-DS}
Let $0<q<1$ and $\ms,\pms\in\RM$.
Suppose that there exists a positive supersolution
$u\in\LS^q_{\rm loc}(\dom,d\ms)\cap\DS^{1,2}_0(\dom)$ of
$-\Delta u=\ms u^q+\pms$ in $\dom$.
Then $-\Delta u\in\DS^{-1,2}(\dom)\cap\RM$ and
$\pms\in\DS^{-1,2}(\dom)$.
\end{lemma}

\begin{lemma}\label{lem:uniqueness}
Assumptions are the same as in Theorem \ref{thm:L^q-sol}.
If $\msol,u\in\CF(\cl)\cap\LS^q(\dom,d\ms)$
are solutions of \eqref{eq:SubLE} such that
$\msol\le u$ in $\dom$,
then $\msol=u$ in $\dom$.
\end{lemma}

\begin{proof}
When $\ms(\dom)=0$,
the equation in \eqref{eq:SubLE} becomes $-\Delta u=\pms$,
and therefore the uniqueness of \eqref{eq:SubLE} is well known.
Let us consider the case $\ms(\dom)>0$.
Since $\msol$ and $u$ are bounded and superharmonic in $\dom$,
it follows from \cite[Corollary 7.20]{MR2305115} that
$\msol,u\in\Sob_{\rm loc}^{1,2}(\dom)$ and
\begin{align}
	\int_\dom \nabla \msol\cdot\nabla\phi \, dx
  &=
	\int_\dom \msol^q\phi \, d\ms+\int_\dom \phi \, d\pms,  \label{eq:weak-u}\\
	\int_\dom \nabla u\cdot\nabla\phi \, dx
  &=
	\int_\dom u^q\phi \, d\ms+\int_\dom \phi \, d\pms  \label{eq:weak-v}
\end{align}
for all $\phi\in\CF_0^\infty(\dom)$.
Note also that $\msol=\bv=u$ on $\bdy$.

Let $\Dom$ be a relatively compact open subset of $\dom$
which is regular for the Dirichlet problem.
Then the restrictions
$\msol,u\in\Sob^{1,2}(\Dom)\cap\CF(\overline{\Dom})$.
Set
\[
	\vmsol := \msol-H_{\msol}^{\Dom}
  \quad \text{and} \quad
	w: = u-H_{u}^{\Dom},
\]
where $H^\Dom_u$ stands for the Perron--Wiener--Brelot solution of \eqref{eq:PWB}
with $\dom=\Dom$ and $f=u$.
Note that $\vmsol$ and $w$ depend on $\Dom$,
as we take limits of those as $\Dom\to\dom$ at the last step.
We see from \cite[Corollary 9.29]{MR2305115} that
$\vmsol,w\in\Sob_0^{1,2}(\Dom)\cap\CF(\overline{\Dom})$
and these are positive in $\Dom$ by the minimum principle.
Moreover, since
$\Delta(\vmsol-w)=\Delta(\msol-u)=\ms(u^q-\msol^q)\ge 0$
in $\Dom$ in the distributional sense,
we have $\vmsol\le w$ in $\Dom$ by the maximum principle.
Also, \eqref{eq:weak-u} and \eqref{eq:weak-v}
yield that for all $\phi\in\CF_0^\infty(\Dom)$,
\begin{align}
	\int_\Dom \nabla \vmsol\cdot\nabla\phi \, dx
  &=
	\int_\Dom \msol^q\phi \, d\ms
  +
	\int_\Dom \phi \, d\pms,  \label{eq:weak-u-2} \\
	\int_\Dom \nabla w\cdot\nabla\phi \, dx
  &=
	\int_\Dom u^q\phi \, d\ms
  +
	\int_\Dom \phi \, d\pms. \label{eq:weak-v-2}
\end{align}
By Lemma \ref{lem:dual-DS},
we see that Radon measures
$\msol^q\,d\ms$, $u^q\,d\ms$ and $\pms$ belong to
$\DS^{-1,2}(\Dom)$.
Thus Lemma \ref{lem:BreB} ensures that
\eqref{eq:weak-u-2} and \eqref{eq:weak-v-2} are valid for all
$\phi\in\DS_0^{1,2}(\Dom)\cap\CF(\Dom)$.
Applying \eqref{eq:weak-u-2} with $\phi=w$
and \eqref{eq:weak-v-2} with $\phi=\vmsol$, we have
\begin{align*}
	\int_\Dom \nabla\vmsol\cdot \nabla w \, dx
  &=
	\int_\Dom \msol^q w \, d\ms + \int_\Dom w \, d\pms,  \\
	\int_\Dom \nabla w\cdot \nabla\vmsol \, dx
  &=
	\int_\Dom u^q \vmsol \, d\ms + \int_\Dom \vmsol \, d\pms.
\end{align*}
Subtracting in each side, we get
\begin{equation}\label{eq:pr-unique-1}
	\int_\Dom (\msol^qw-u^q \vmsol) \, d\ms
  =
	\int_\Dom (\vmsol-w) \, d\pms \le 0.
\end{equation}
By the definitions of $\vmsol$ and $w$, we see that 
\begin{equation}\label{eq:pr-unique-5}
	\msol^q w-u^q \vmsol
  = 
	\msol^q u^q(u^{1-q} - \msol^{1-q}) 
  +
	H^{\Dom}_u (u^q-\msol^q) - u^q(H^{\Dom}_u - H^{\Dom}_{\msol}) \qtext{in } \Dom.
\end{equation}
Thus, by \eqref{eq:pr-unique-1} and \eqref{eq:pr-unique-5},
\[
	\int_{\Dom} \msol^q u^q(u^{1-q} - \msol^{1-q})  \, d\ms
  +
	\int_{\Dom} H^{\Dom}_u (u^q-\msol^q) \, d\ms
  \le
	\int_{\Dom} u^q (H^{\Dom}_u-H^{\Dom}_{\msol}) \,d\ms.
\]
Note that
$0\le H^{\Dom}_u(u^q-\msol^q)\le\|u\|_\infty u^q$ and
$0\le u^q(H^{\Dom}_u-H^{\Dom}_{\msol})\le\|u-\msol\|_\infty u^q$
in $\Dom$, and $u^q\in\LS^1(\dom,d\ms)$ by assumption.
Since
$H_{\msol}^{\Dom} \to H_{\msol}=H_{\bv}$ and
$H_{u}^{\Dom} \to H_{u}=H_{\bv}$
as $\Dom$ expands to $\dom$
(see \cite[Theorem 6.3.10]{MR1801253}),
it follows from
the monotone convergence and the Lebesgue dominated convergence
theorems that
\[
	\int_{\dom} \msol^q u^q (u^{1-q} - \msol^{1-q}) \,d\ms  
  +
	\int_{\dom} H_{\bv}(u^q-\msol^q)\, d\ms
  \le  0.
\]
Both integrands on the left hand side are nonnegative. 
Therefore $u=\msol$ $\ms$-a.e. in $\dom$.
By the integral representations \eqref{eq:int-eq} of $\msol$ and $u$,
we conclude that $u=\msol$ in $\dom$.
This completes the proof.
\end{proof}

We need the following characterization of weighted norm inequalities
established by Verbitsky.

\begin{lemma}[{\cite[Theorem 1.1]{MR3792109}}]\label{lem:normineq}
Let $p>1$, let $0<r<p$ and let $\ms\in\RM$.
Then there exists a constant $\con>0$ such that
\[
	\|\GF[\psi \, d\ms]\|_{\LS^r(\dom,d\ms)}
  \le
	\con \|\psi\|_{\LS^p(\dom,d\ms)}
  \qtext{for } \psi\in\LS^p(\dom,d\ms)
\]
if and only if $\GF[\ms]\in\LS^{\frac{pr}{p-r}}(\dom,d\ms)$.
\end{lemma}

We are now ready to prove Theorem \ref{thm:L^q-sol}.

\begin{proof}[Proof of Theorem \ref{thm:L^q-sol}]
The necessity follows immediately from \eqref{eq:int-eq} and
the lower estimate in Theorem \ref{thm:upper-lower-est}.
Let us show the sufficiency.
As in the proof of Theorem \ref{thm:existence}
in Section \ref{subsec:ex-2},
we construct a nondecreasing sequence $\{u_j\}$
converging to a solution $u\in\CF(\cl)$ of \eqref{eq:SubLE}.
We claim that $\{u_j\}\subset\LS^\gamma(\dom,d\ms)$.
In fact, $u_0\in\LS^\gamma(\dom,d\ms)$ by assumption.
If $u_{j-1}\in\LS^\gamma(\dom,d\ms)$, then
Lemma \ref{lem:normineq} with $r=\gamma$ and $p=\frac{\gamma}{q}$ gives
\[
\begin{split}
	\|u_j\|_{\LS^\gamma(\dom,d\ms)}
  &\le
	\|\GF[u_{j-1}^q\,d\ms]\|_{\LS^\gamma(\dom,d\ms)}
  +
	\|\GF[\pms]+H_{\bv}\|_{\LS^\gamma(\dom,d\ms)}  \\
  &\lesssim
	\|u_{j-1}\|_{\LS^\gamma(\dom,d\ms)}^q
  +
	\|\GF[\pms]+H_{\bv}\|_{\LS^\gamma(\dom,d\ms)}<+\infty.
\end{split}
\]
Thus the claim follows by induction.
Moreover, since $u_{j-1}\le u_j$, we have in the same manner as in \eqref{eq:upper-est-1}
\[
	\|u_j\|_{\LS^\gamma(\dom,d\ms)}
  \lesssim
	(1+\|\GF[\pms]+H_{\bv}\|_{\LS^\gamma(\dom,d\ms)})^{\frac{1}{1-q}}
  <+\infty.
\]
Letting $j\to+\infty$, we get $u\in\LS^\gamma(\dom,d\ms)$,
as desired.

Finally, we show the uniqueness.
As mentioned in Remark \ref{rem:minimalsol},
problem \eqref{eq:SubLE} has a minimal solution
$\msol\in\CF(\cl)\cap\LS^\gamma(\dom,d\ms)$.
Let
$\msol\in\CF(\cl)\cap\LS^\gamma(\dom,d\ms)$
be any solution of \eqref{eq:SubLE}.
Then $\msol\le u$ in $\dom$.
If $\gamma=q$, then
$\msol=u$ in $\dom$ by Lemma \ref{lem:uniqueness}.
On the other hand, if $\gamma=q+1$ and $\bv\equiv 0$, then
$\msol,u\in\DS^{1,2}_0(\dom)$
by \cite[Theorem 1.1]{MR3881877}.
Therefore
$\msol=u$ in $\dom$
by their continuity and \cite[Theorem 6.3]{MR4048382}.
This completes the proof.
\end{proof}

\section{Remarks}\label{sec:remark}

In this section, we suppose that the dimension $\di\ge3$.

\subsection{H\"older continuous solution}

Let $0<\alpha<1$.
By $\CF^{0,\alpha}(\dom)$,
we denote the set of all $\alpha$-H\"older continuous functions in $\dom$.
Also, $\CF^{0,\alpha}_{\rm loc}(\dom)$
stands for the set of all functions $u\in\CF^{0,\alpha}(E)$
for every compact subset $E$ of $\dom$.

\begin{corollary}
Let $0<q<1$, let $\bv\in\CF_+(\bdy)$ and let
$\ms,\pms\in\RM$ satisfy
\eqref{eq:katocond-2} and \eqref{eq:cond-f-meas}.
Assume that there exist $\alpha\in(0,1)$ and $\conlabel{c:holder}>0$
such that the inequality
\begin{equation}\label{eq:mu<r}
	\omega(B(x,r)\cap\dom) \le \conref{c:holder} r^{\di-2+\alpha}
  \qtext{for all } x\in\dom \text{ and } r>0
\end{equation}
holds for both of $\omega=\ms$ and $\omega=\pms$.
Then there exists a solution
$u\in\CF_{\rm loc}^{0,\alpha}(\dom)\cap\CF(\cl)$ of
\eqref{eq:SubLE}.
\end{corollary}

\begin{proof}
Using the formula:
\[
	\int_{B(x,r)}\|x-y\|^{2-\di} \, d\omega(y)
  =
	r^{2-\di} \omega(B(x,r))
  + (\di-2) \int_0^r t^{1-\di} \omega(B(x,t)) \, dt,
\]
and $\GF(x,y)\lesssim\|x-y\|^{2-\di}$,
we see from \eqref{eq:mu<r} that
$\ms$ and $\pms$ satisfy \eqref{eq:katocond-1}.
By Theorem \ref{thm:existence},
there exists a solution $u\in\CF(\cl)$ of \eqref{eq:SubLE}.
Let $\Dom$ be a relatively compact open subset of $\dom$
and let $\omega:=u^q \ms|_{\Dom}+\pms|_{\Dom}$.
Then the local Riesz decomposition theorem for superharmonic functions shows that
there exists a positive harmonic function $h$ on $\Dom$ such that
$u(x)=h(x)+G_{\R^\di}[\omega](x)$ for all $x\in\Dom$.
By the way, we have for all $x\in \R^\di$ and $r>0$,
\[
	\omega(B(x,r))
  \le
	\|u\|_{\infty}^q \ms(B(x,r)\cap\Dom)
	+ \pms(B(x,r)\cap\Dom)
  \le
	\con  r^{\di-2+\alpha},
\]
where a constant $\con$ depends only on
$\conref{c:holder}$, $q$, $f$, $\ms$, $\pms$.
Therefore it follows from \cite[Lemma 6.1]{MR2783397} that
$G_{\R^\di}[\omega]\in\CF^{0,\alpha}(\overline{\Dom})$,
and so $u\in\CF_{\rm loc}^{0,\alpha}(\Dom)$.
\end{proof}

\subsection{Solution whose Dirichlet integral diverges}

Recall that
if $\dom$ is a bounded Lipschitz domain in $\R^\di$,
then there exist positive constants $\con$ and $\beta\le1$
such that
\begin{equation}\label{eq:Green-est}
	\GF(x,y)
  \le
	\con
	\left( \frac{\dist(x)}{\|x-y\|} \right)^\beta
	\left( \frac{\dist(y)}{\|x-y\|} \right)^\beta \|x-y\|^{2-\di}
  \qtext{for all } x,y\in\dom.
\end{equation}
See \cite[Section 2]{MR3877278} for example.
If $\dom$ has a $\CF^{1,1}$-boundary,
then we can take $\beta=1$.
For $\alpha\in\R$, let
\[
	d\ms_\alpha(x) := \dist(x)^{-\alpha}\, dx.
\]

\begin{lemma}\label{lem:kato-1}
Let $\dom$ be a bounded Lipschitz domain in $\R^\di$ $(\di\ge3)$
and let $\beta$ be as above.
If $\alpha<1+\beta$, then
$\mu_\alpha$ satisfies the $\GF$-Kato condition.
\end{lemma}

\begin{proof}
Let $x\in\dom$ and $r>0$.
Put $\rho(x):=\min\{r,\dist(x)/2\}$.
Then
\begin{align*}
	&\int_{\dom\cap B(x,r)} \GF(x,y) \, d\ms_\alpha(y)  \\
  &\qquad=
	\int_{\dom\cap B(x,\rho(x))}\GF(x,y) \dist(y)^{-\alpha} \, dy
  +
	\int_{\dom\cap B(x,r)\setminus B(x,\rho(x))}\GF(x,y) \dist(y)^{-\alpha} \, dy  \\
  &\qquad=:
	I_1+I_2.
\end{align*}
If $\rho(x)=r$, then $I_2=0$ and
\[
	I_1
  \lesssim
	\int_{\dom\cap B(x,\rho(x))}\|x-y\|^{2-\di} \dist(y)^{-\alpha} \, dy
  \lesssim
	\dist(x)^{-\alpha} r^2
  \lesssim
  \begin{cases}
	r^2 & \text{if } \alpha\le 0,  \\
	r^{2-\alpha} & \text{if } \alpha>0.
  \end{cases}
\]
We consider the case $\rho(x)=\dist(x)/2$.
Then
$I_1\lesssim\dist(x)^{2-\alpha}\lesssim r^{2-\alpha}$.
Let $N$ be the smallest natural number such that
$r\le 2^N\dist(x)$
and take $\eta\in\partial\dom$ so that
$\|\eta-x\|=\dist(x)$.
Then $B(x,2^k\dist(x))\subset B(\eta,2^{k+1}\dist(x))$
for every integer $k\ge0$.
Recall that if $\tau>-1$, then
\[
	\int_{\dom\cap B(\xi,R)} \dist(x)^{\tau} \, dx
  \lesssim
	R^{\di+\tau}
  \qtext{for } \xi\in\partial\dom \text{ and } R>0.
\]
These, together with \eqref{eq:Green-est}, yield that
\begin{align*}
	I_2
  &\lesssim
	\int_{\dom\cap B(x,r)\setminus B(x,\rho(x))} \|x-y\|^{2-\di-\beta} \dist(y)^{\beta-\alpha} \, dy  \\
  &\lesssim
	\sum_{k=0}^N (2^k\dist(x))^{2-\di-\beta}
	\int_{\dom\cap B(x,2^k\dist(x))\setminus B(x,2^{k-1}\dist(x))}
	\dist(y)^{\beta-\alpha} \, dy  \\
  &\lesssim
	\sum_{k=0}^N (2^k\dist(x))^{2-\di-\beta}
	(2^{k+1}\dist(x))^{\di+(\beta-\alpha)}  \\
  &\lesssim
	\dist(x)^{2-\alpha} \sum_{k=0}^N 2^{k(2-\alpha)}
  \lesssim
	r^{2-\alpha}.
\end{align*}
Therefore, in any cases,
$\ms_\alpha$ satisfies the $\GF$-Kato condition
\eqref{eq:katocond-1}.
\end{proof}

\begin{lemma}\label{lem:finiteenergy}
Assume that $\dom$ is a bounded $\CF^{1,1}$-domain in $\R^\di$ $(\di\ge3)$.
Let $0<q<1$, let $\gamma>0$ and let $\alpha\in\R$.
Then problem
\begin{equation}\label{eq:homo-sublinear-0}
  \begin{cases}
	-\Delta u=\ms_\alpha u^q & \text{in } \dom,  \\
	u>0 & \text{in } \dom,  \\
	u=0 & \text{on } \partial\dom,
  \end{cases}
\end{equation}
has a solution $u\in\LS_{\rm loc}^q(\dom)$ with
\begin{equation}\label{eq:Dirichletint-3}
	\int_\dom \|\nabla u\|^2u^{\gamma-1} \, dx<+\infty
\end{equation}
if and only if
\[
	\alpha < \frac{2\gamma+1+q}{\gamma+1}.
\]
\end{lemma}

\begin{proof}
In light of \cite[Theorem 1.1]{MR3881877},
it suffices to show that
\[
	\int_\dom \GF[\ms_\alpha](x)^{\frac{\gamma+q}{1+q}} \, d\ms_\alpha(x)
  < +\infty
  \quad \Longleftrightarrow \quad
	\alpha < \frac{2\gamma+1+q}{\gamma+1}.
\]

($\Longrightarrow$)\:
Suppose to the contrary that $\alpha\ge(2\gamma+1+q)/(\gamma+1)$.
Then
\[
	(2-\alpha)\frac{\gamma+q}{1-q}-\alpha \le -1.
\]
Since
\[
	\GF[\ms_\alpha](x)
  \gtrsim
	\int_{\dom\cap B(x,\dist(x)/2)} \|x-y\|^{2-\di} \dist(y)^{-\alpha} \, dy
  \gtrsim
	\dist(x)^{2-\alpha}
\]
for all $x\in\dom$, we have
\begin{align*}
	\int_\dom \GF[\ms_\alpha](x)^{\frac{\gamma+q}{1-q}} \, d\ms_\alpha(x)
  \gtrsim
	\int_\dom \dist(x)^{(2-\alpha)\frac{\gamma+q}{1-q}-\alpha} \, dx
  = + \infty.
\end{align*}

($\Longleftarrow$)\:
Note that $\alpha<2$.
Let $x\in\dom$.
We split $\GF[\ms_\alpha](x)=I_1+I_2$, where
\begin{align*}
	I_1 &:= \int_{B(x,\dist(x)/2)} \GF(x,y)\,d\ms_\alpha(y),  \\
	I_2 &:= \int_{\dom\setminus B(x,\dist(x)/2)} \GF(x,y)\,d\ms_\alpha(y).
\end{align*}
It is easy to see that $I_1\lesssim \dist(x)^{2-\alpha}$.
Also, in a manner similar to the proof of Lemma \ref{lem:kato-1},
we can estimate $I_2$ as follows.
Let $N$ be the smallest natural number such that $\diam\dom\le2^N\dist(x)$.
Then, by \eqref{eq:Green-est} with $\beta=1$,
\begin{align*}
	I_2
  &\lesssim
	\sum_{k=0}^N (2^k\dist(x))^{-\di}\dist(x)
	\int_{B(x,2^k\dist(x))}\dist(y)^{1-\alpha} \, dy  \\
  &\lesssim
	\dist(x)^{2-\alpha} \sum_{k=0}^N 2^{k(1-\alpha)}  \\
  &\lesssim
	F_\alpha(x):=
  \begin{cases}
	\dist(x)^{2-\alpha} & \text{if } \alpha>1,  \\
	\dist(x)^{2-\alpha}\log\dfrac{2\diam\dom}{\dist(x)} & \text{if } \alpha=1, \\
	\dist(x) & \text{if } \alpha<1.
  \end{cases}
\end{align*}
Therefore $\GF[\ms_\alpha](x)\lesssim F_\alpha(x)$, and so
\[
	\int_\dom \GF[\ms_\alpha](x)^{\frac{\gamma+q}{1+q}} \, d\ms_\alpha(x)
  \lesssim
	\int_\dom F_\alpha(x)^{\frac{\gamma+q}{1+q}} \dist(x)^{-\alpha} \, dx
  < +\infty.
\]
Thus the lemma is proved.
\end{proof}

\begin{corollary}\label{exm:kato-2}
Assume that $\dom$ is a bounded $\CF^{1,1}$-domain in $\R^\di$ $(\di\ge3)$.
Let $0<q<1$, let $\gamma>0$ and let
\[
	\frac{2\gamma+1+q}{\gamma+1} \le \alpha<2.
\]
Then \eqref{eq:homo-sublinear-0}
has a solution $u\in\CF^1(\dom)\cap\CF(\overline{\dom})$,
but fails to have a solution
$u\in\LS_{\rm loc}^q(\dom)$ with \eqref{eq:Dirichletint-3}.
\end{corollary}

\begin{proof}
By Lemma \ref{lem:kato-1} and Theorem \ref{thm:existence},
there exists a solution $u\in\CF(\overline{\dom})$ of \eqref{eq:homo-sublinear-0}.
Since $\ms_\alpha u^q\in\LS_{\rm loc}^\infty(\dom)$,
we observe that $u\in\CF^1(\dom)$.
The second assertion follows from Lemma \ref{lem:finiteenergy}.
\end{proof}

\section*{Acknowledgements}
Part of this work was written while the second named author was at the Department of Mathematics, Hokkaido University. 
He would like to express his  gratitude to this institution for the hospitality and support provided.

\end{document}